\documentclass{amsart}

\usepackage{amssymb}
\usepackage{graphicx, epsfig}
\usepackage{latexsym, amsfonts, amscd, amsmath}
\usepackage[all]{xy}

\theoremstyle{plain}
\newtheorem{thm}{Theorem}[section]
\newtheorem{prop}[thm]{Proposition}
\newtheorem{lem}[thm]{Lemma}
\newtheorem{cor}[thm]{Corollary}

\theoremstyle{definition}

\newtheorem{dfn}[thm]{Definition}

\theoremstyle{remark}

\newtheorem{rmk}[thm]{Remark}

\catcode`\Ž=\active\def Ž{\'e}
\catcode`\=\active\def {\`e}
\catcode`\ˆ=\active\def ˆ{\`a}
\catcode`\‰=\active\def ‰{\^a}
\catcode`\Ë=\active\def Ë{\`A}
\catcode`\Š=\active\def Š{\"a}
\catcode`\=\active\def {\`u}
\catcode`\=\active\def {\^{e}}
\catcode`\"=\active\def "{\^{i}}
\catcode`\•=\active\def •{\"i}
\catcode`\™=\active\def ™{\^{o}}
\catcode`\š=\active\def š{\"{o}}
\catcode`\ž=\active\def ž{\^{u}}
\catcode`\Ÿ=\active\def Ÿ{\"{u}}
\catcode`\†=\active\def †{\"{U}}
\catcode`\=\active\def {\c{c}}
\catcode`\'=\active\def '{\c{C}}
\catcode`\é=\active\def é{\c{C}}

\newcommand{\Hom}{{\operatorname{Hom}}}

\newcommand{\ord}{{\operatorname{ord }}}

\newcommand{\SL}{{\operatorname{SL }}}
\newcommand{\GL}{{\operatorname{GL}}}

\newcommand{\smallmat}[4]{\bigl(\begin{smallmatrix}#1&#2\\#3&#4\end{smallmatrix}\bigr)}

\newcommand{\gert}{{\mathfrak{t}}}

\newcommand{\gerH}{{\mathfrak{H}}}

\newcommand{\calC}{{\mathcal{C}}}

\newcommand{\calX}{{\mathcal{X}}}

\def\C{\mathbb{C}}
\def\D{\mathbb{D}}

\def\M{\mathbb{M}}

\def\Q{\mathbb{Q}}
\def\R{\mathbb{R}}

\def\X{\mathbb{X}}
\def\Z{\mathbb{Z}}

\def\W{\mathbb W}
\def\Y{\mathbb Y}
\def\X{\mathbb X}
\def\mat#1#2#3#4{\left( \begin{array}{cc} #1 & #2 \\ #3 & #4 \end{array} \right) }

\def\2vector#1#2{\left( \begin{smallmatrix} #1 \\ #2 \end{smallmatrix}
\right)}

\def\deb{ \begin{equation} }
\def\fin{ \end{equation} }

\usepackage[usenames]{color}
\definecolor{Indigo}{rgb}{0.2,0.1,0.7}
\definecolor{Violet}{rgb}{0.5,0.1,0.7}
\definecolor{White}{rgb}{1,1,1}
\definecolor{Green}{rgb}{0.1,0.9,0.2}

\newcommand{\longmono}{\mbox{$\lhook\joinrel\longrightarrow$}}

\begin{document}

\title[Shimura-Shintani-Waldspurger]{The $\Lambda$-adic Shimura-Shintani-Waldspurger correspondence}
\date{\today}
\author{Matteo Longo, Marc-Hubert Nicole}

\begin{abstract} 
We generalize the $\Lambda$-adic Shintani lifting for $\GL_2(\Q)$ to indefinite quaternion algebras over $\Q$.

\end{abstract}

\address{Dipartimento di Matematica, Universit\`a di Padova, Via Trieste 63, 35121 Padova, Italy}
\email{mlongo@math.unipd.it}

\address{Institut de mathŽmatiques de Luminy, UniversitŽ d'Aix-Marseille, campus de Luminy, case 907, 13288 Marseille cedex 9, France}
\email{nicole@iml.univ-mrs.fr}

\maketitle

\section{Introduction}

Langlands's principle of functoriality predicts the existence of a staggering wealth of transfers (or lifts) between automorphic forms for different reductive groups. In recent years, attempts at the formulation of $p$-adic variants of Langlands's functoriality have been articulated in various special cases.  We prove the existence of the Shimura-Shintani-Waldspurger lift for $p$-adic families.

More precisely, Stevens, building on the work of Hida and Greenberg-Stevens, showed in \cite{St} the existence of a $\Lambda$-adic variant of the classical Shintani lifting of \cite{Shin} for $\GL_2(\Q)$. 
This $\Lambda$-adic lifting can be seen as a formal power series with coefficients in a finite extension of the Iwasawa algebra $\Lambda:=\Z_p[\![X]\!]$  
equipped with specialization maps interpolating classical Shintani 
lifts of classical modular forms appearing in a given Hida family. 

Shimura in \cite{Sh}, resp. Waldspurger in \cite{Wa} generalized the classical Shimura-Shintani correspondence to quaternion algebras over $\Q$, resp.
over any number field. In the $p$-adic realm, Hida (\cite{Hida3}) constructed a $\Lambda$-adic Shimura lifting, while Ramsey (\cite{Ra}) (resp. Park \cite{Pa}) extended the Shimura (resp. Shintani) lifting to the overconvergent setting.

In this paper, motivated by ulterior applications to Shimura curves over $\Q$, we generalize Stevens's result to any non-split rational indefinite quaternion algebra $B$, building on work of Shimura 
\cite{Sh}
and combining this with a result of Longo-Vigni \cite{LV1}. 
Our main result, for which the reader is referred to Theorem \ref{main} below, states the existence 
of a formal power series and specialization maps interpolating Shimura-Shintani-Waldspurger lifts of classical forms in a given $p$-adic family of automorphic forms
on the quaternion algebra $B$. 
The $\Lambda$-adic variant of Waldspurger's result appears computationally challenging (see remark in \cite[Intro.]{Pr}), but it seems within reach for real quadratic fields (cf. \cite{Po}). 

As an example of our main result, we consider the case of families with trivial character. 
Fix a prime number $p$ and a positive integer $N$ such that $p\nmid N$. Embed the set $\Z^{\geq 2}$ of integers greater or equal to 2 
in $\Hom(\Z_p^\times,\Z_p^\times)$ by sending $k\in\Z^{\geq 2}$ to the character $x\mapsto x^{k-2}$. 
Let $f_\infty$ be an Hida family of tame level $N$
passing through a form $f_0$ of level $\Gamma_0(Np)$ and weight $k_0$.  
There is a neighborhood $U$ of $k_0$ in $\Hom(\Z_p^\times,\Z_p^\times)$
such that, 
for any $k\in\Z^{\geq 2}\cap U$, the weight $k$ specialization of 
$f_\infty$ gives rise to an element 
$f_{k}\in S_{k}(\Gamma_0(Np))$. 
Fix a factorization $N=MD$ with $D>1$ a square-free product of an even number of primes and $(M,D)=1$
(we assume that such a factorization exists). 
Applying the Jacquet-Langlands correspondence 
we get for any $k\in\Z^{\geq 2}\cap U$ a modular form $f_k^{\rm JL}$ on $\Gamma$, which is the group of  
norm-one elements in an Eichler order $R$ of level $Mp$ contained in the indefinite rational 
quaternion algebra $B$ of discriminant $D$. One can show that these modular forms can be $p$-adically 
interpolated, up to scaling, in a neighborhood of $k_0$.
More precisely, let $\mathcal O$ be the ring of integers of a finite extension $F$ 
of $\Q_p$ and let $\D$ denote the $\mathcal O$-module of $\mathcal O$-valued measures on $\Z_p^2$ which 
are supported on the set of primitive elements in $\Z_p^2$. Let 
$\Gamma_0$ be the group of norm-one elements in an Eichler order $R_0\subseteq B$ 
containing $R$. There is a canonical action of $\Gamma_0$ on $\D$ 
(see \cite[\S 2.4]{LV1} for its description). Denote by $F_k$ the extension of $F$ generated  by the 
Fourier coefficients of $f_k$. 
Then there is an element $\Phi\in H^1(\Gamma_0,\D)$ 
and maps 
\[\rho_k:H^1(\Gamma_0,\D)\longrightarrow H^1(\Gamma,F_k)\] 
such that $\rho(k)(\Phi)=\phi_k$, the cohomology class associated to $f_k^{\rm JL}$, 
with $k$ in a neighborhood of $k_0$ 
(for this we need a suitable normalization of the cohomology class associated to $f_k^{\rm JL}$, which 
we do not touch for simplicity in this introduction). 
We view $\Phi$ as a quaternionic family of modular forms. 
To each $\phi_k$ we may apply the Shimura-Shintani-Waldspurger lifting (\cite{Sh}) and obtain a modular form $h_k$ of weigh $k+1/2$, 
level $4Np$ and  trivial character. We show that this collection of forms can be $p$-adically interpolated. For clarity's sake, we present the liftings and their $\Lambda$-adic variants in a diagram, in which the horizontal maps are specialization maps of the $p$-adic family to weight $k$; JL stands for the Jacquet-Langlands correspondence; SSW stands for the Shimura-Shintani-Waldspurger lift; and the dotted arrows are constructed in this paper:
\[\xymatrix{{f_\infty} \ar@{|->}[r]^{} \ar@{|->}[d]_{\Lambda-\text{adic JL}} & f_{k} \ar@{|->}[d]^{\text{JL}} \\
\Phi \ar@{|->}[r]^{\rho_k} \ar@{|..>}[d]_{\Lambda-\text{adic SSW}} & \phi_k \ar@{|->}[d]^{\text{SSW}} \\
\Theta  \ar@{|..>}[r]^{} & h_k 
}\]
More precisely, as a particular case of our main result, Theorem \ref{main}, we get the following

\begin{thm} \label{thm-intro} There exists a $p$-adic neighborhood $U_0$ 
of $k_0$ in $\Hom(\Z_p^\times,\Z_p^\times)$, $p$-adic periods $\Omega_k$ for $k\in U_0\cap\Z^{\geq 2}$  
and a formal expansion 
\[\Theta=\sum_{\xi\geq 1} a_\xi q^\xi\] 
with coefficients $a_\xi$ in the ring of $\C_p$-valued functions on $U_0$, 
such that for all $k\in U_0\cap \Z^{\geq 2}$ we have 
\[\Theta(k)=\Omega_k\cdot h_k .\] Further, $\Omega_{k_0}\neq 0$. \end{thm}

\section{Shintani integrals and Fourier coefficients of half-integral weight modular forms}\label{sec2}

We express the Fourier coefficients of half-integral weight modular forms 
in terms of period integrals, thus allowing a cohomological interpretation which is key to the production of the $\Lambda$-adic version of the Shimura-Shintani-Waldspurger correspondence.
For the quaternionic Shimura-Shintani-Waldspurger correspondence of interest to us (see \cite{Pr}, \cite{Wa}), the period integrals expressing the values of the Fourier coefficients have been computed generally by Prasanna in \cite{Pr2}. 
 
\subsection{The Shimura-Shintani-Waldspurger lifting}\label{sec-SSL} 

Let $4M$ be a positive integer, $2k$ an even non-negative integer 
and $\chi$ a Dirichlet character modulo $4M$ such that $\chi(-1) = 1$.
Recall that the space of half-integral weight modular forms $S_{k+1/2}(4M, \chi)$ consists of holomorphic cuspidal functions $h$ on the upper-half place $\gerH$ such that 
\[ h( \gamma(z)) = j^{1/2} (\gamma, z)^{2k+1} \chi(d) h(z) ,\]
\noindent for all $\gamma = \smallmat {a}{b}{c}{d} \in \Gamma_0(4M)$, where ${j^{1/2}}(\gamma,z)$ is the standard square root of the usual automorphy factor $j(\gamma,z)$ (cf. \cite[2.3]{Pr}).

To any quaternionic integral weight modular form we may associate a half-integral weight modular form following Shimura's work \cite{Sh}, 
as we will describe below.   

Fix an odd square free integer $N$ and a factorization $N=M\cdot D$ into coprime integers 
such that $D>1$ is a product of an even number of distinct primes. 
Fix a Dirichlet character $\psi$ modulo $M$ and a positive even integer $2k$. Suppose that 
\[\psi(-1)=(-1)^k.\]
Define the Dirichlet character $\chi$ modulo $4N$ by 
\[\chi(x):=\psi(x)\left(\frac{-1}{x}\right)^{k}.\]

Let $B$ be an indefinite quaternion algebra over $\Q$ of discriminant
$D$. Fix a maximal order $\mathcal O_B$ of $B$.  
For every prime $\ell | M$, choose an isomorphism 
\[i_\ell:B\otimes_\Q\Q_\ell\simeq\M_2(\Q_\ell)\] such that
$i_\ell(\mathcal O_B\otimes_\Z\Z_\ell)=\M_2(\Z_\ell)$. Let 
$R\subseteq \mathcal O_B$ be the Eichler order of $B$ of 
level $M$ defined by requiring that $i_\ell(R\otimes_\Z\Z_\ell)$ is the suborder of $\M_2(\Z_\ell)$ 
of upper triangular matrices modulo $\ell$ for all $\ell | M$.
Let $\Gamma$ 
denote the subgroup of the group $R^\times_1$ of norm $1$ elements in $R^\times$
consisting of those $\gamma$ such that $i_\ell(\gamma)\equiv \smallmat1*01$ mod $\ell$ for all $\ell| M$. 
We denote by $S_{2k}(\Gamma)$ the $\C$-vector space of weight $2k$ modular forms on $\Gamma$,  
and by $S_{2k}(\Gamma,\psi^2)$ the subspace of $S_{2k}(\Gamma)$ consisting of forms
having character $\psi^2$ under the action of $R_1^\times$. 
Fix a Hecke eigenform \[f\in S_{2k}(\Gamma,\psi^2)\] as in \cite[Section 3]{Sh}.

Let $V$ denote the $\Q$-subspace of $B$ consisting of elements with 
trace equal to zero. For any $v\in V$, which we view as a trace zero 
matrix in $\M_2(\R)$ (after fixing an isomorphism $i_\infty:B\otimes\R\simeq\M_2(\R)$), set 
\[G_v:=\{\gamma\in\SL_2(\R)|\, \gamma^{-1}v\gamma=v\}\] and put $\Gamma_v:=G_v\cap\Gamma$. 
One can show that there exists an isomorphism $\omega:\R^\times\overset\sim\rightarrow {G_v}$ defined by 
$\omega(s)=\beta^{-1}\smallmat s00{s^{-1}}\beta$, for some $\beta\in\SL_2(\R)$.
Let ${\gert}_v$ be the order of $\Gamma_v \cap \left\{ \pm 1 \right\}$ and let $\gamma_v$ be an element of 
$\Gamma_v$ which generates $\Gamma_v \left\{ \pm 1 \right\} / \left\{ \pm 1 \right\}$. 
Changing $\gamma_v$ to $\gamma^{-1}_v$ if necessary, we may assume 
$\gamma_v=\omega(t)$ with $t>0$. 
Define $V^*$ to be the $\Q$-subspace of $V$ 
consisting of elements with strictly negative norm. 
For any $\alpha=\smallmat a{b}{c}{-a}\in V^*$ and $z\in\mathcal H$, define the quadratic form  
\[Q_\alpha(z):=cz^2-2az-b.\] Fix $\tau\in\mathcal H$ and set 
\[ {P(f,\alpha,\Gamma) := - \big(2 (-{\rm nr}(\alpha))^{1/2} / {\gert}_\alpha\big) \int_\tau^{\gamma_\alpha(\tau)} Q_\alpha(z)^{k-1}f(z) dz} \]
where ${\rm nr}:B\rightarrow\Q$ is the norm map.
By \cite[Lemma 2.1]{Sh},
the integral is independent on the choice $\tau$, which justifies the notation.

\begin{rmk}
The definition of  
$P(f,\alpha,\Gamma)$ given in \cite[(2.5)]{Sh} looks different: the above expression can be derived 
as in \cite[page 629]{Sh} by means of \cite[(2.20) and (2.22)]{Sh}.   
\end{rmk}

Let $R(\Gamma)$ denote the set of equivalence classes 
of $V^*$ under the action of $\Gamma$ by conjugation.  By \cite[(2.6)]{Sh}, $P(f,\alpha,\Gamma)$ only depends on the 
conjugacy class of $\alpha$, and thus, for $\mathcal C\in R(\Gamma )$, we may 
define $P(f,\calC,\Gamma):=P(f,\alpha,\Gamma)$
for any choice of $\alpha\in\mathcal C$.  Also,
$q(\calC):=-{\rm nr}(\alpha)$ for any $\alpha\in\mathcal C$. 

Define $\mathcal O'_B$ to be the maximal order in $B$ such that $\mathcal O'_B\otimes_\Z\Z_\ell\simeq
\mathcal O_B\otimes_\Z\Z_\ell$ for all $\ell\nmid M$ and $\mathcal O'_B\otimes_\Z\Z_\ell$ is equal 
to the local order of $B\otimes_\Q\Q_\ell$ consisting of elements $\gamma$ such that $i_\ell(\gamma)=
\smallmat a{b/M}{cM}{d}$ with $a,b,c,d\in\Z_\ell$, for all $\ell| M$. 
Given $\alpha\in\mathcal O_B'$, we can find an integer $b_\alpha$ such that
\begin{equation}\label{eq-b}
i_\ell(\alpha)\equiv \mat {*}{b_\alpha/M}{*}{*} \mod i_\ell(R\otimes_\Z\Z_\ell), \quad\forall\ell| M.\end{equation}
Define a locally constant function $\eta_\psi$ on $V$ by $\eta_\psi(\alpha)=\psi(b_\alpha)$ if $\alpha\in \mathcal O_B'\cap V$ 
and $\eta(\alpha)=0$ otherwise, with $\psi(a)=0$ if $(a,M)\neq 1$ 
(for the definition of {locally constant functions on $V$} 
in this context, we refer to \cite[p. 611]{Sh}). 

For any $\mathcal C\in R(\Gamma )$, 
fix $\alpha_\mathcal C\in \mathcal C$. For any integer $\xi\geq 1$, define 
\[\
a_\xi(\tilde h):= \big(2\mu(\Gamma\backslash\mathfrak H)\big)^{-1}\cdot\sum_{\mathcal C\in R(\Gamma), q(\calC)=\xi} \eta_\psi(\alpha_\calC)
\xi^{-1/2} P(f,\calC,\Gamma).\] 
Then, by \cite[Theorem 3.1]{Sh}, 
\[
\tilde h:=\sum_{\xi\geq 1}a_\xi(\tilde h)q^\xi \in S_{k+1/2}(4N,\chi)\]
is called the {Shimura-Shintani-Waldspurger lifting} of $f$. 
 
\subsection{Cohomological interpretation} \label{sec2.3}
We introduce necessary notation to define the action of the Hecke action on cohomology groups; for details, 
see \cite[\S 2.1]{LV1}. If $G$ is a subgroup of $B^\times$ and $S$ a subsemigroup of $B^\times$ such 
that $(G,S)$ is an Hecke pair, we let $\mathcal H(G,S)$ 
denote the Hecke algebra corresponding to $(G,S)$, whose elements are written as $T(s)=GsG=\coprod_iG{s_i}$ for $s,s_i\in S$
(finite disjoint union). For any $s\in S$, let $s^*:={\rm norm}(s)s^{-1}$ and denote by $S^*$ the set of elements of the form 
$s^*$ for $s\in S$. For any $\Z[S^*]$-module $M$ we let $T(s)$ act on $H^1(G,M)$ at the level of cochains $c\in Z^1(G,M)$ 
by the formula $(c|T(s))(\gamma)=\sum_i s_i^* c(t_i(\gamma))$, where $t_i(\gamma)$ are defined by the equations 
$Gs_i\gamma=Gs_j$ and $s_i\gamma=t_i(\gamma)s_j$. In the following, we will consider the case of 
$G=\Gamma$ and 
\[\text{$S=\{s\in B^\times| i_\ell(s)$ is congruent to $\smallmat{1}{*}{0}{*}$ mod $\ell$ for all $\ell| M$\}}.\]

For any field $L$ and any integer $n\geq 0$, let  $V_{n}(L)$ denote the $L$-dual of the $L$-vector space 
$\mathcal P_{n}(L)$
of homogeneous 
polynomials in 2 variables of degree $n$.  We let $\M_2(L)$ act from the right on $P(x,y)$ as $P|\gamma(x,y):=P(\gamma(x,y))$, where for $\gamma=\smallmat abcd$ we put 
\[\gamma(x,y):=(ax+yb,cx+dy).\] This also equips $V_n(L)$ with a left action 
by $\gamma\cdot\varphi(P):=\varphi(P|\gamma)$.  To simplify the notation, we will write $P(z)$ for $P(z,1)$. 

Let $F$ denote the finite extension of $\Q$ generated by the eigenvalues of the Hecke action on $f$. 
For any field $K$ containing $F$, set 
\[ \W_f(K):=H^1\bigl(\Gamma,V_{k-2}(K)\bigr)^{f}\] 
where the superscript $f$ denotes the subspace on which the Hecke algebra acts via the character associated with $f$. 
Also, for any sign $\pm$, let $\W_f^\pm(K)$ denote the $\pm$-eigenspace for the action of the archimedean involution
$\iota$. Remember that $\iota$ is defined by choosing an element $\omega_\infty$ of norm $-1$ in $R^\times$ such that 
such that $i_\ell(\omega_\infty)\equiv \smallmat 100{-1}$ mod $M$ for all primes $\ell| M$ and 
then setting $\iota:=T(w_\infty)$ 
(see \cite[\S 2.1]{LV1}).
Then $\W_f^\pm(K)$ is one dimensional (see, {e.g.}, \cite[Proposition 2.2]{LV1}); fix a generator $\phi^\pm_f$ 
of $\W_f^\pm(F)$. 

To explicitly describe $\phi_f^\pm$, let us introduce some more notation. 
Define 
\[f|\omega_\infty(z):=(Cz+D)^{-k/2}\overline{f(\omega_\infty(\bar z))}\] 
where $i_\infty(\omega_\infty)=\smallmat ABCD$.
Then $f|\omega_\infty\in S_{2k}(\Gamma)$ as well.  
If the eigenvalues of the Hecke action on $f$ are real, then we may assume, after multiplying $f$ 
by a scalar, that 
$f|\omega_\infty=f$ (see \cite[p. 627]{Sh} or \cite[Lemma 4.15]{LV2}). In general,
let  
 $I(f)$ denote the class in 
$H^1(\Gamma,V_{k-2}(\C))$ represented by the cocycle 
\[\gamma\longmapsto \left[P\mapsto I_\gamma(f)(P):=\int_{\tau}^{\gamma(\tau)}f(z)P(z)dz\right]\]
for any $\tau\in \mathcal H$ (the corresponding class is independent on the choice of $\tau$). 
With this notation, 
\[ P(f,\alpha,\Gamma)=- \big(2 (-{\rm nr}(\alpha))^{1/2} / {\gert}_\alpha\big) \cdot I_{\gamma_{\alpha_\mathcal C}}(f)\big(Q_{\alpha_\mathcal C}(z)^{k-1}\big)
.\]
Denote by $I^\pm(f):=(1/2)\cdot I(f)\pm (1/2)\cdot I(f)|\omega_\infty$, the projection of $I(f)$ 
to the eigenspaces for the action of $\omega_\infty$. Then $I(f)=I^+(f)+I^-(f)$
and $I_f^\pm=\Omega_f^\pm\cdot\phi_f^\pm$, for some $\Omega_f^\pm\in\C^\times$. 

Given $\alpha\in V^*$ of norm $-\xi$, 
put $\alpha':=\omega_\infty^{-1}\alpha\omega_\infty$. 
By \cite[4.19]{Sh}, we have 
\[{\eta(\alpha)}{\xi^{-1/2}}P(f,\alpha,\Gamma)+{\eta(\alpha')}{\xi^{-1/2}}P(f,\alpha',\Gamma)=
-{\eta(\alpha)}\cdot{{\gert}_\alpha}^{-1}\cdot I^+_{\gamma_{\alpha}}\big(Q_{\alpha_\mathcal C}(z)^{k-1}\big).
\] We then have 
\[a_\xi(\tilde h)=\sum_{\mathcal C\in R_2(\Gamma), q(\calC)=\xi} \frac{-\eta_\psi(\alpha_\calC)}
{2\mu(\Gamma\backslash\mathcal H)\cdot \gert_{\alpha_\mathcal C}}
\cdot I^+_{\gamma_{\alpha_\mathcal C}}\big(Q_{\alpha_\mathcal C}(z)^{k-1}\big).
\]

We close this section by choosing a suitable multiple of $h$ which will be the object of the next section. 
Given $Q_\alpha(z)=cz^2-2az-b$ as above, 
with $\alpha$ in $V^*$, define $\tilde Q_\alpha(z):=M \cdot Q_\alpha(z)$.  Then, clearly,  
$I^\pm(f)(\tilde Q_{\alpha_\mathcal C}(z)^{k-1})$ is equal to $M^{k-1}I^\pm(f)(Q_{\alpha_\mathcal C}(z)^{k-1})$. We thus normalize the Fourier coefficients by setting 
\begin{equation}\label{SS1}a_\xi(h):=-\frac{a_\xi(\tilde h)\cdot M^{k-1}\cdot2\mu(\Gamma\backslash\mathcal H)}
{\Omega_f^{+}}=
\sum_{\mathcal C\in R(\Gamma), q(\calC)=\xi} \frac{\eta_\psi(\alpha_\calC)}{\gert_{\alpha_\calC}}
\cdot\phi_f^{+}\big(\tilde Q_{\alpha_\mathcal C}(z)^{k-1}\big).\end{equation}
So \begin{equation}\label{SS2}h:=\sum_{\xi\geq 1} a_\xi(h)q^\xi\end{equation}
belongs to $S_{k+1/2}(4N,\chi)$ and is a non-zero multiple of $\tilde h$. 

\section{The $\Lambda$-adic Shimura-Shintani-Waldspurger correspondence}
At the heart of Stevens's proof lies the control theorem of Greenberg-Stevens, which has been worked out in the quaternionic setting by Longo--Vigni \cite{LV1}. 

Recall that $N\geq 1$ is a square free integer and fix a decomposition 
$N=M\cdot D$ where $D$ is a square free product of an even number of primes and $M$ is coprime to $D$. 
Let $p\nmid N$ be a prime number and fix an embedding $\bar\Q\hookrightarrow\bar\Q_p$.  

\subsection{The Hida Hecke algebra}\label{sec-Hida}
Fix an ordinary $p$-stabilized newform 
\begin{equation}\label{f}
f_0\in S_{k_0}\big(\Gamma_1(Mp^{r_0})\cap\Gamma_0(D),\epsilon_0\big)\end{equation} of level $\Gamma_1(Mp^{r_0})\cap\Gamma_0(D)$, Dirichlet character $\epsilon_0$ and weight $k_0$, and write $\mathcal O$ for the ring of integers of the field 
generated over $\Q_p$ by the Fourier coefficients of $f_0$. 

Let $\Lambda$ (respectively, ${{\mathcal O[\![\Z_p^\times]\!]}}$) denote the Iwasawa algebra of $W:=1+p\Z_p$ (respectively, $\Z_p^\times$) 
with coefficients in $\mathcal O$. We denote group-like elements in $\Lambda$ and ${{\mathcal O[\![\Z_p^\times]\!]}}$ as 
$[t]$. 
Let $\mathfrak h_\infty^\ord$ denote the $p$-ordinary Hida Hecke algebra with coefficients in $\mathcal O$ 
of tame level $\Gamma_1(N)$. Denote by
$\mathcal L:={\rm Frac}(\Lambda)$ the fraction field of $\Lambda$. 
Let $\mathcal R$ denote the integral closure of $\Lambda$ in the primitive component $\mathcal K$ of 
${\mathfrak h_\infty^{\ord}}\otimes_\Lambda\mathcal L$ corresponding to $f_0$. It is well 
known that the $\Lambda$-algebra $\mathcal R$ is finitely generated as $\Lambda$-module. 

Denote by $\mathcal X $ the $\mathcal O$-module
$\Hom^{\rm cont}_{\mathcal O\text{-alg}}(\mathcal R,\bar\Q_p)$ 
of continuous homomorphisms of $\mathcal O$-algebras.  
Let $\mathcal X^{\rm arith} $ the set of arithmetic homomorphisms in $\mathcal X $, 
defined in \cite[\S 2.2]{LV1} by requiring that the composition 
\[W\longmono\Lambda\overset\kappa\longrightarrow\bar\Q_p\]
has the form $\gamma\mapsto\psi_\kappa(\gamma)\gamma^{n_\kappa}$ with $n_\kappa=k_\kappa-2$ 
for an
integer $k_\kappa\geq 2$ (called the {weight of $\kappa$}) and a finite order character 
$\psi_\kappa:W\rightarrow\bar\Q_p$ (called the {wild character of $\kappa$}). Denote by 
$r_\kappa$ the smallest among the positive integers $t$ such that $1+p^t\Z_p\subseteq\ker(\psi_\kappa)$. 
For any $\kappa\in\mathcal X^{\rm arith} $, let $P_\kappa$ denote the kernel of $\kappa$ and $\mathcal R_{P_\kappa}$ 
the localization of $\mathcal R$ at $\kappa$. The field $F_\kappa:=\mathcal R_{P_\kappa}/P_\kappa\mathcal R_{P_\kappa}$ 
is a finite extension of ${\rm Frac}(\mathcal O)$. Further, by duality, $\kappa$ corresponds to a normalized eigenform 
\[f_\kappa\in S_{k_\kappa}\big(\Gamma_0(Np^{r_\kappa}),\epsilon_\kappa\big)\]
for a Dirichlet character $\epsilon_\kappa:(\Z/Np^{r_\kappa}\Z)^\times\rightarrow\bar\Q_p$. 
More precisely, if we write $\psi_\mathcal R$ for the character of $\mathcal R$, defined as in 
\cite[Terminology p. 555]{Hida-Galois}, and we let $\omega$ denote the 
TeichmŸller character, we have 
$\epsilon_\kappa:=\psi_\kappa\cdot\psi_\mathcal R\cdot \omega^{-n_\kappa}$
(see \cite[Cor. 1.6]{Hida-Galois}). We call $(\epsilon_\kappa,k_\kappa)$ the  signature of $\kappa$. 
We let $\kappa_0$ denote the 
arithmetic character associated with $f_0$, so $f_0=f_{\kappa_0}$, $k_0=k_{\kappa_0}$, 
$\epsilon_0=\epsilon_{\kappa_0}$, 
and $r_0=r_{\kappa_0}$. 
The eigenvalues of $f_\kappa$ under the action of the Hecke operators 
$T_n$ ($n\geq 1$ an integer) belong to $F_\kappa$.  Actually, one can 
show that $f_\kappa$ is a $p$-stabilized newform on $\Gamma_1(Mp^{r_\kappa})\cap\Gamma_0(D)$. 

Let $\Lambda_N$ denote the Iwasawa algebra of $\Z_p^\times\times (\Z/N\Z)^\times$ with coefficients in $\mathcal O$. 
To simplify the notation, define $\Delta:= (\Z/Np\Z)^\times$. We have a canonical isomorphism of rings $\Lambda_N\simeq \Lambda[\Delta]$, which makes $\Lambda_N$ a $\Lambda$-algebra, finitely generated as $\Lambda$-module. 
Define the tensor product of $\Lambda$-algebras  
\[{\mathcal R}_N:=\mathcal R\otimes_\Lambda\Lambda_N,\] which is again a $\Lambda$-algebra (resp. $\Lambda_N$-algebra)
finitely generated as a $\Lambda$-module, (resp. as a $\Lambda_N$-module). 
One easily checks that 
there is a canonical isomorphism of $\Lambda$-algebras
\[\mathcal R_N\simeq
\mathcal R[\Delta]\] (where $\Lambda$ acts on $\mathcal R$); this is also an isomorphism 
of $\Lambda_N$-algebras, when we let $\Lambda_N\simeq\Lambda[\Delta]$ 
act on $\mathcal R[\Delta]$ in the obvious way. 

We can extend any $\kappa\in\mathcal X^{\rm arith} $ to a continuous $\mathcal O$-algebra morphism 
\[\kappa_N:{\mathcal R}_N\longrightarrow\bar\Q_p\] setting 
\[\kappa_N\left(\sum_{i=1}^nr_i\cdot\delta_i\right):=
\sum_{i=1}^n\kappa(r_i)\cdot\psi_\mathcal R(\delta_i)\] for $r_i\in\mathcal R$ and $\delta_i\in\Delta$.
Therefore, $\kappa_N$ restricted to $\Z_p^\times$ is the character $t\mapsto\epsilon_\kappa(t)t^{n_\kappa}$. 
If we denote by $\mathcal X_N $ the $\mathcal O$-module 
of continuous $\mathcal O$-algebra homomorphisms from $\mathcal R_N$ to $\bar\Q_p$, the above correspondence sets 
up an injective map $\mathcal X^{\rm arith} \hookrightarrow\mathcal X_N $. 
Let $\mathcal X_N^{\rm arith} $ denote the 
image of $\mathcal X^{\rm arith} $ under this map. For $\kappa_N\in\mathcal X^{\rm arith}_N $, we define the  signature 
of $\kappa_N$ to be that of the corresponding $\kappa$. 

\subsection{The control theorem in the quaternionic setting} \label{CT}
Recall that $B/\Q$ is a quaternion algebra of discriminant $D$. 
Fix an auxiliary real quadratic field 
$F$ such that all primes dividing $D$ are inert in $F$ and all primes dividing $Mp$ 
are split in $F$, and an isomorphism $i_F:B\otimes_\Q F\simeq \M_2(F)$. Let $\mathcal O_B$ 
denote the maximal order of $B$ obtained by taking the intersection of $B$ with 
$\M_2(\mathcal O_F)$, where $\mathcal O_F$ is the ring of integers of $F$. More precisely, 
define 
\[\mathcal O_B:=\iota^{-1}\big( i_F^{-1}\big(i_F(B\otimes 1)\cap\M_2(\mathcal O_F)\big)\big)\] 
where $\iota:B\hookrightarrow B\otimes_\Q F$ is the inclusion defined by $b\mapsto b\otimes 1$. This is 
a maximal order in $B$ because $i_F(B\otimes 1)\cap\M_2(\mathcal O_F)$ is a maximal order 
in $i_F(B\otimes 1)$. In particular, $i_F$ and our fixed embedding of $\bar\Q$ into $\bar\Q_p$ induce 
an isomorphism \[i_p: B\otimes_\Q\Q_p\simeq\M_2(\Q_p)\]
such that $i_p(\mathcal O_B\otimes_\Z\Z_p)=\M_2(\Z_p)$. For any prime $\ell| M$, also choose an embedding $\bar\Q\hookrightarrow\bar\Q_\ell$ which, 
composed with $i_F$, yields isomorphisms 
\[i_\ell: B\otimes_\Q\Q_\ell\simeq\M_2(\Q_\ell)\] such that $i_p(\mathcal O_B\otimes_\Z\Z_\ell)=\M_2(\Z_\ell)$.
Define an Eichler order $R\subseteq\mathcal O_B$ of level $M$ by requiring that for all primes 
$\ell| M$ the image of $R\otimes _\Z\Z_\ell$ via $i_\ell$ consists of upper triangular matrices modulo $\ell$. 
For any $r\geq 0$, let $\Gamma_r$ denote the subgroup of the group $R_1^\times$ of norm-one elements in 
$R$ consisting of those $\gamma$ such that $i_\ell(\gamma)=\smallmat abcd$ with $c\equiv 0\mod Mp^r$ 
and $a\equiv d\equiv  1\mod Mp^r$, for all primes $\ell| Mp$. 
To conclude this list of notation and definitions, fix an embedding 
$F\hookrightarrow \R$ and let \[i_\infty:B\otimes_\Q\R\simeq\M_2(\R)\] be the induced isomorphism. 

Let $\Y:=\Z_p^2$ and denote by $\X$ the set of primitive vectors in $\Y$. Let $\D$ denote the $\mathcal O$-module of 
$\mathcal O$-valued 
measures on $\Y$ which are supported on $\X$. 
 Note that 
$\M_2(\Z_p)$ acts on $\Y$ by left multiplication; this induces an action of $\M_2(\Z_p)$ on the $\mathcal O$-module of 
$\mathcal O$-valued measures on $\Y$, which induces an action on $\D$. 
The group $R^\times$ acts on $\D$ via $i_p$. In particular, we may define the group:   
\[ \W:=H^1(\Gamma_0,\D).\] 
Then $\D$ has a canonical structure of ${{\mathcal O[\![\Z_p^\times]\!]}}$-module, as well as $\mathfrak h_\infty^\ord$-action, 
as described in \cite[\S 2.4]{LV1}. In particular, let us recall that, for any $[t]\in{{\mathcal O[\![\Z_p^\times]\!]}}$, we have 
\[\int_{\X}\varphi(x,y)d\big([t]\cdot\nu\big)=\int_\X\varphi(tx,ty)d\nu,\] for any locally constant function 
$\varphi$ on $\X$.

For any $\kappa\in\mathcal X^{\rm arith} $ and any sign $\pm \in \left\{ -,+\right\}$, set 
 \[ \W_\kappa^\pm:=\W_{f_\kappa^{\rm JL}}^\pm(F_\kappa)=H^1\bigl(\Gamma_{r_\kappa},V_{n_\kappa}(F_\kappa)\bigr)^{f_\kappa,\pm}\] where $f_\kappa^{\rm JL}$ is any Jacquet-Langlands lift of $f_\kappa$ to $\Gamma_{r_\kappa}$; 
recall that the superscript $f_\kappa$ 
denotes the subspace on which the Hecke algebra acts via the character associated with $f_\kappa$, and the superscript 
$\pm$ denotes the $\pm$-eigenspace for the action of the archimedean involution
$\iota$.
Also, recall that $\W_\kappa^\pm$ is one dimensional and fix a generator $\phi^\pm_{\kappa}$ of it. 

We may define {specialization maps}  
\[ \rho_{\kappa}:\D\longrightarrow V_{n_\kappa}(F_\kappa) \]
by the formula
\begin{equation}\label{rho-kappa}
\rho_{\kappa}(\nu)(P):=\int_{\Z_p\times\Z_p^\times}\epsilon_\kappa(y)P(x,y)d\nu\end{equation}
which induces (see \cite[\S 2.5]{LV1}) a map: 
\[\rho_\kappa:\W^\ord\longrightarrow \W^\ord_{\kappa}.\] 
Here $\W^\ord$ and $\W^\ord_\kappa$ denote the {ordinary submodules} of $\W$ and $\W_\kappa$, respectively, 
defined as in \cite[Definition 2.2]{GS} (see also \cite[\S3.5]{LV1}).
We also let $\W_\mathcal R:=\W\otimes_\Lambda\mathcal R$, and extend the above map $\rho_\kappa$ to a map
\[\rho_\kappa:\W_\mathcal R^\ord\longrightarrow \W^\ord_\kappa\] 
by setting $\rho_\kappa(x\otimes r):=\rho_\kappa(x)\cdot \kappa(r)$.

\begin{thm} \label{thm-LV}
There exists a $p$-adic neighborhood $\mathcal U_0 $ of $\kappa_0$ in $\mathcal X $,  
elements $\Phi^\pm$ in $\W^\ord_\mathcal R$ and choices of $p$-adic periods ${\Omega_\kappa^\pm\in F_\kappa}$ 
for $\kappa\in \mathcal U_0 \cap \mathcal X^{\rm arith} $ 
such that, for all $\kappa\in \mathcal U_0\cap \mathcal X^{\rm arith} $, we have \[ \rho_\kappa(\Phi^\pm) = \Omega_\kappa^\pm \cdot\phi_\kappa^\pm\] 
and $\Omega_{\kappa_0}^\pm\neq 0$.\end{thm}
\begin{proof}
This is an easy consequence of \cite[Theorem 2.18]{LV1} 
and follows along the lines 
of the proof of \cite[Theorem 5.5]{St}, cf. \cite[Proposition 3.2]{LV2}.\end{proof}

We now normalize our choices as follows. With $\mathcal U_0 $ as above, define 
\[\mathcal U^{\rm arith}_0 :=\mathcal U_0 \cap\mathcal X^{\rm arith} .\] 
Fix $\kappa\in\mathcal U^{\rm arith}_0 $ and  
an embedding $\bar\Q_p\hookrightarrow\C$. 
Let 
$f_\kappa^{\rm JL}$ denote a modular form on $\Gamma_{r_\kappa}$ corresponding to $f_\kappa$ 
by the Jacquet-Langlands correspondence, which is well defined up to elements in $\C^\times$. View 
$\phi_\kappa^\pm$ as an element in $H^1(\Gamma_{r_\kappa},V_n(\C))^\pm$. 
Choose a representative $\Phi_\gamma^\pm$ of $\Phi^\pm$, by which 
we mean that if $\Phi^\pm=\sum_i\Phi_i^\pm\otimes r_i$, then we choose 
a representative $\Phi^\pm_{i,\gamma}$ for all $i$. 
Also, we will write $\rho_\kappa(\Phi)(P)$ as 
\[\int_{\Z_p\times\Z_p^\times} \epsilon_\kappa(y)P(x,y)d\Phi_\gamma^\pm:=
\sum_i\kappa(r_i)\cdot
\int_{\Z_p\times\Z_p^\times} \epsilon_\kappa(y)P(x,y)d\Phi_{i,\gamma}^\pm.\]
With this notation, we see that the two cohomology classes
\[\gamma\longmapsto\int_{\Z_p\times\Z_p^\times} \epsilon_\kappa(y)P(x,y)d\Phi_\gamma^\pm(x,y)\]
and 
\[\gamma\longmapsto\Omega_\kappa^\pm\cdot
\int_{\tau}^{\gamma(\tau)}P(z,1)f_\kappa^{{\rm JL},\pm}(z)dz\] are cohomologous in 
$H^1(\Gamma_{r_\kappa},V_{n_\kappa}(\C))$, for any choice of $\tau\in\mathcal H$. 

\subsection{Metaplectic Hida Hecke algebras}\label{metaplectic}
Let $\sigma: {{\Lambda}_N} \rightarrow {{\Lambda}_N}$ be the ring homomorphism associated to the group homomorphism $t \mapsto t^2$ on $\Z_p^\times\times(\Z/N\Z)^\times$, and denote by the same symbol its restriction to 
$\Lambda$ and $\mathcal O[\![\Z_p^\times]\!]$. 
We let $\Lambda_\sigma$,  $\mathcal O[\![\Z_p^\times]\!]_\sigma$ 
and $ {\Lambda}_{N,\sigma}$ denote, respectively, 
$\Lambda$, $\mathcal O[\![\Z_p^\times]\!]$
and ${\Lambda}_{N}$ viewed as algebras over themselves via $\sigma$. 
The ordinary metaplectic $p$-adic Hida Hecke algebra we will consider is the $\Lambda$-algebra 
\[{\widetilde{\mathcal R}}:=
\mathcal R\otimes_{\Lambda} \Lambda_\sigma.\]

Define as above 
\[{\widetilde{\mathcal X} }:=\Hom^{\rm cont}_{\mathcal O\text{-alg}}({\widetilde{\mathcal R}},\bar\Q_p)\]
and let the set $\widetilde{\mathcal X}^{\rm arith} $ of arithmetic points in 
${\widetilde{\mathcal X} }$ to consist of those $\tilde\kappa$ such that the composition 
\[\xymatrix{
W\ar@{^(->}[r]& 
\Lambda \ar@{^(->}[rr]^-{\lambda\mapsto 1\otimes\lambda}&&
{\widetilde{\mathcal R}}\ar@{^(->}[r]^-{\tilde\kappa}&\bar\Q_p}\]
has the form $\gamma\mapsto\psi_{\tilde\kappa}(\gamma)\gamma^{n_{\tilde\kappa}}$ with 
$n_{\tilde\kappa}:=k_{\tilde\kappa}-2$ 
for an  integer $k_{\tilde\kappa}\geq 2$ (called the {weight of $\tilde\kappa$}) and a finite order character 
$\psi_{\tilde\kappa}:W\rightarrow\bar\Q$ (called the {wild character of $\tilde\kappa$}). Let $r_{\tilde\kappa}$ 
the smallest among the positive integers $t$ such that $1+p^t\Z_p\subseteq\ker(\psi_{\tilde\kappa})$. 

We have a  
map $p:{\widetilde{\mathcal X} }\rightarrow \mathcal X $ induced by pull-back from the canonical 
map $\mathcal R\rightarrow {\widetilde{\mathcal R}}$. The map $p$ 
restricts to arithmetic points.

As above, define the $\Lambda$-algebra (or $\Lambda_N$-algebra) 
\begin{equation}\label{R_N-def}
\widetilde{\mathcal R}_{N}:={\mathcal R}\otimes_\Lambda\Lambda_{N,\sigma}\end{equation}
via $\lambda\mapsto 1\otimes\lambda$.  

We easily see that 
\[\widetilde{\mathcal R}_N\simeq \widetilde{\mathcal R}[\Delta]\] as $\Lambda_N$-algebras, 
where we enhance $\widetilde{\mathcal R}[\Delta]$  with the following structure of $\Lambda_N\simeq\Lambda[\Delta]$-algebra: 
for $\sum_{i}\lambda_i\cdot\delta_i\in\Lambda[\Delta]$ (with $\lambda_i\in\Lambda$ and 
$\delta_i\in\Delta$) and $\sum r_j\cdot\delta'_j\in\widetilde{\mathcal R}[\Delta]$ (with $r_j=\sum_h r_{j,h}\otimes\lambda_{j,h}
\in\widetilde{\mathcal R}$, $r_{j,h}\in\mathcal R$, $\lambda_{j,h}\in\Lambda_\sigma$, 
and $\delta'_j\in\Delta$), 
we set \[\big(\sum_i\lambda_i\cdot\delta_i\big)\cdot\big(\sum_jr_j\cdot\delta'_j\big):=\sum_{i,j,h}
\big(r_{j,h}\otimes(\lambda_i\lambda_{j,h})\big)\cdot(\delta_i\delta'_j).\] 

As above, extend 
$\tilde\kappa\in\widetilde{\mathcal X}^{\rm arith} $ to a continuous $\mathcal O$-algebra morphism 
$\tilde\kappa_N:\widetilde{\mathcal R}_{N}\rightarrow\bar\Q_p$ by setting 
\[\tilde\kappa_N\left(\sum_{i=1}^nx_i\cdot\delta_i\right):=
\sum_{i=1}^n\tilde\kappa(x_i)\cdot\psi_\mathcal R(\delta_i)\] for $x_i\in\widetilde{\mathcal R}$ and $\delta_i\in\Delta$, where $\psi_\mathcal R$ is the character of $\mathcal R$. 
If we denote by $\widetilde{\mathcal X}_N $ the $\mathcal O$-module 
of continuous $\mathcal O$-linear homomorphisms from $\widetilde{\mathcal R}_N$ to $\bar\Q_p$, the above correspondence sets 
up an injective map $\widetilde{\mathcal X}^{\rm arith} \hookrightarrow\widetilde{\mathcal X}_N $
and we let $\widetilde{\mathcal X}_N^{\rm arith}  $ denote the 
image of $\widetilde{\mathcal X}^{\rm arith} $. Put   
$\epsilon_{\tilde\kappa}:=\psi_{\tilde\kappa}\cdot \psi_\mathcal R\cdot\omega^{-n_{\tilde\kappa}}$, which we view 
as a Dirichlet character of $(\Z/Np^{r_{\tilde\kappa}}\Z)^\times$, 
and call the pair $(\epsilon_{\tilde\kappa},k_{\tilde\kappa})$ the  signature of $\tilde\kappa_N$, where $\tilde\kappa$ is the 
arithmetic point corresponding to $\tilde\kappa_N$.  

We also have a
map $p_N:{\widetilde{\mathcal X}_N }\rightarrow \mathcal X_N $ induced from the
map $\mathcal R_N\rightarrow {\widetilde{\mathcal R}}_{N}$ taking 
$r \mapsto r\otimes1$ by pull-back.
The map $p_N$ also restricts to 
arithmetic points. The maps $p$ and $p_N$ make the following diagram commute:
\[\xymatrix{\widetilde{\mathcal X}^{\rm arith} \ar@{^(->}[r]\ar[d]^-{p} & \widetilde{\mathcal X}^{\rm arith}_N \ar[d]^-{p_N}\\
\mathcal X^{\rm arith} \ar@{^(->}[r] & \mathcal X^{\rm arith}_N 
}\] where the projections take a signature $(\epsilon,k)$ to $(\epsilon^2,2k)$.

\subsection{The $\Lambda$-adic correspondence}
In this part, we combine the explicit integral formula of Shimura and the fact that the toric integrals can be $p$-adically interpolated to show the existence of a $\Lambda$-adic Shimura-Shintani-Waldspurger correspondence with the expected interpolation property. This follows very closely \cite[\S 6]{St}.

Let 
$\tilde\kappa_N\in\widetilde{\mathcal X}_N^{\rm arith}$ 
of signature 
$(\epsilon_{\tilde\kappa},k_{\tilde\kappa})$. 
Let $L_r$ denote the order of $\M_2(F)$ consisting of matrices $\smallmat a{b/{Mp^r}}{Mp^rc}{d}$ 
with $a,b,c,d\in\mathcal O_F$.  Define
\[\mathcal O_{B,r}:=\iota^{-1}\big( i_F^{-1}\big(i_F(B\otimes 1)\cap L_r\big)\big)\] 
Then $\mathcal O_{B,r}$ 
is the maximal order introduced in \S\ref{sec-SSL} (and denoted $\mathcal O_B'$ there)
defined in terms of the maximal order $\mathcal O_B$ and the integer $Mp^r$. 
Also, 
$S:=\mathcal O_B\cap \mathcal O_{B,r}$ is an Eichler order of $B$ of level $Mp$ containing the fixed Eichler order $R$ of 
level $M$.  
With $\alpha\in V^*\cap \mathcal O_{B,1}$, we have 
\begin{equation}\label{b}
i_F(\alpha)=\mat a{b/(Mp)}c{-a}\end{equation} 
in $\M_2(F)$ with 
$a,b,c\in\mathcal O_F$ and we can 
consider the quadratic forms   
\[
Q_\alpha(x,y):=cx^2-2axy-\big(b/(Mp)\big)y^2,\]
and 
\begin{equation}\label{Q}\tilde Q_\alpha(x,y):=Mp\cdot Q_\alpha(x,y)=Mpcx^2-2Mpaxy-by^2.\end{equation}
Then $\tilde Q_\alpha(x,y)$ 
has coefficients in $\mathcal O_F$ and, composing with $F\hookrightarrow\R$ 
and letting $x=z$, $y=1$, we recover $Q_\alpha(z)$ and $\tilde Q_\alpha(z)$  of \S\ref{sec-SSL} 
(defined by means of the isomorphism $i_\infty$).  
Since each prime $\ell| Mp$ is split in $F$, the elements $a,b,c$ can be viewed 
as elements in $\Z_\ell$ via our fixed embedding $\bar\Q\hookrightarrow\bar\Q_\ell$, 
for any prime $\ell| Mp$ (we will continue writing $a,b,c$ for these elements, with a slight abuse 
of notation). So, letting $b_\alpha\in\Z$ such that $i_\ell(\alpha)\equiv\smallmat *{b_\alpha/(Mp)}{**}{*}$ modulo 
$i_\ell(S\otimes_\Z\Z_\ell)$, for all $\ell| Mp$, 
we have $b\equiv b_\alpha$ modulo $Mp\Z_\ell$ as elements in $\Z_\ell$, for all $\ell| Mp$,  
and thus we get 
\begin{equation}\label{eta}
\eta_{\epsilon_{\tilde\kappa}}(\alpha)=\epsilon_{\tilde\kappa}(b_\alpha)=\epsilon_{\tilde\kappa}(b)
\end{equation} for $b$ as in \eqref{b}.  

For any $\nu\in\D$, we may define an $\mathcal O$-valued measure $j_\alpha(\nu)$  on $\Z_p^\times$ 
by the formula: 
\[\int_{\Z_p^\times}f(t)dj_\alpha(\nu)(t):=\int_{\Z_p\times\Z_p^\times}f\big(\tilde Q_\alpha(x,y)\big)d\nu(x,y).\]
for any continuous function $f:\Z_p^\times\rightarrow\mathcal \C_p$. Recall that the group of $\mathcal O$-valued 
measures on $\Z_p^\times$ is isomorphic to the Iwasawa algebra ${{\mathcal O[\![\Z_p^\times]\!]}}$, and thus 
we may view $j_\alpha(\nu)$ as an element in ${{\mathcal O[\![\Z_p^\times]\!]}}$ (see, for example, \cite[\S 3.2]{CS}). 
In particular, for 
any group-like element $[\lambda]\in{{\mathcal O[\![\Z_p^\times]\!]}}$ we have:  
\[\int_{\Z_p^\times}f(t)d\big([\lambda]\cdot j_\alpha(\nu)\big)(t)=
\int_{\Z_p^\times}\left(\int_{\Z_p^\times}f(ts)d[\lambda](s)\right)d j_\alpha(\nu)(t)=
\int_{\Z_p^\times}f(\lambda t)d j_\alpha(\nu)(t).\] On the other hand, 
\[\int_{\Z_p\times\Z_p^\times}f\big(\tilde Q_\alpha(x,y)\big)d(\lambda\cdot\nu)=
\int_{\Z_p\times\Z_p^\times}f\big(\tilde Q_\alpha(\lambda x,\lambda y)\big)d\nu=
\int_{\Z_p\times\Z_p^\times}f\big(\lambda^2 \tilde Q_\alpha(x,y)\big)d\nu
\]
and we conclude that $j_\alpha(\lambda\cdot\nu)=[\lambda^2]\cdot j_{\alpha}(\nu)$. In other words, $j_\alpha$ 
is a $\mathcal O[\![\Z_p^\times]\!]$-linear map
\[j_\alpha:\D\longrightarrow {{\mathcal O[\![\Z_p^\times]\!]}}_\sigma.\]

Before going ahead, let us introduce some notation. 
Let $\chi$ be a Dirichlet character modulo $Mp^r$, for a positive integer $r$, 
which we decompose accordingly with the isomorphism 
$(\Z/Np^r\Z)^\times\simeq (\Z/N\Z)^\times\times(\Z/p^r\Z)^\times$ into the product $\chi=\chi_N\cdot\chi_p$ with 
$\chi_N:(\Z/N\Z)^\times\rightarrow \C^\times$ and 
$\chi_p:(\Z/p^r\Z)^\times\rightarrow \C^\times$. Thus, we will write $\chi(x)=\chi_N(x_N)\cdot\chi_p(x_p)$, 
where $x_N$ and $x_p$ are the projections of 
$x\in(\Z/Np^r\Z)^\times$ to $(\Z/N\Z)^\times$ and $(\Z/p^r\Z)^\times$, respectively. To simplify the notation, 
we will suppress the $N$ and $p$ from the notation for $x_N$ and $x_p$, thus simply writing $x$ for any of the two. 
Using the isomorphism 
$(\Z/N\Z)^\times\simeq(\Z/M\Z)^\times\times(\Z/D\Z)^\times$, decompose $\chi_N$ as $\chi_N=\chi_M\cdot\chi_D$
with $\chi_M$ and $\chi_D$ characters
on $(\Z/M\Z)^\times$ and $(\Z/D\Z)^\times$, respectively. In the following, 
we only need  the case when $\chi_D=1$.

Using the above notation, we may define a $\mathcal O[\![\Z_p^\times]\!]$-linear map
$J_\alpha:\D\rightarrow {{\mathcal O[\![\Z_p^\times]\!]}}$
by \[
J_\alpha(\nu)=\epsilon_{\tilde\kappa,M}(b)
\cdot\epsilon_{\tilde\kappa,p}(-1)\cdot j_\alpha(\nu)\]
with $b$ as in \eqref{b}. 
Set $\D_N:=\D\otimes_{\mathcal O[\![\Z_p^\times]\!]}\Lambda_N$, where the map 
${\mathcal O[\![\Z_p^\times]\!]}\rightarrow\Lambda_N$ is induced 
from the map $\Z_p^\times\rightarrow \Z_p^\times\times(\Z/N\Z)^\times$ 
on group-like elements given by  $x\mapsto x\otimes 1$. Then $J_\alpha$ can be 
extended to a $\Lambda_N$-linear 
map $J_\alpha:\D_N\rightarrow\Lambda_{N,\sigma}$. Setting 
$\D_{\mathcal R_N}:=\mathcal R_N\otimes_{\Lambda_N}\D_N$ and extending by 
$\mathcal R_N$-linearity over $\Lambda_N$ we finally obtain a $\mathcal R_N$-linear map, again denoted by the same symbol, 
\[J_\alpha:\D_{\mathcal R_N}\longrightarrow \widetilde{\mathcal R}_N.\]
For $\nu\in\D_N$ and $r\in \mathcal R_N$ we thus have 
\[J_\alpha(r\otimes\nu)= \epsilon_{\tilde\kappa,M}(b)
\cdot\epsilon_{\tilde\kappa,p}(-1)\cdot r\otimes j_\alpha(\nu).\]

For the next result, for any arithmetic point $\kappa_N\in\mathcal X_N^{\rm arith}$ 
coming from $\kappa\in\mathcal X^{\rm arith} $, 
extend $\rho_\kappa$ in \eqref{rho-kappa} by $\mathcal R_N$-linearity over $\mathcal O[\![\Z_p^\times]\!]$, 
to get a map \[\rho_{\kappa_N}:\D_{\mathcal R_N}\longrightarrow V_{n_\kappa}\] defined by 
$\rho_{\kappa_N}(r\otimes \nu):=\rho_\kappa(\nu)\cdot\kappa_N(r)$, for $\nu\in\D$ and $r\in\mathcal R_N$. 
To simplify the notation, set
\begin{equation}\label{pairing}
 \langle \nu,\alpha\rangle_{\kappa_N}  :=  \rho_{\kappa_N}(\nu)(\tilde Q_\alpha^{n_{\tilde\kappa}/2})  .\end{equation}
The following is essentially \cite[Lemma (6.1)]{St}.

\begin{lem}\label{lemma3.2} Let 
$\tilde\kappa_N\in\widetilde{\mathcal X}_N^{\rm arith}  $ with signature $(\epsilon_{\tilde\kappa},k_{\tilde\kappa})$ 
and define $\kappa_N:=p_N(\tilde\kappa_N)$.  
Then for 
any $\nu\in \D_{\mathcal R_N}$ we have: 
\[\tilde\kappa_N\big(J_\alpha(\nu)\big)=\eta_{\epsilon_{\tilde\kappa}}(\alpha)\cdot\langle \nu,\alpha\rangle_{\kappa_N}.\] 
\end{lem}
\begin{proof} For $\nu\in\D_N$ and $r\in \mathcal R_N$ we have 
\[\begin{split}\tilde\kappa_N\big(J_\alpha(r\otimes\nu)\big)&=
\tilde\kappa_N\big(\epsilon_{\tilde\kappa,M}(b)
\cdot\epsilon_{\tilde\kappa,p}(-1)\cdot r\otimes j_\alpha(\nu)\big)\\
&=\epsilon_{\tilde\kappa,M}(b)
\cdot\epsilon_{\tilde\kappa,p}(-1)\cdot\tilde\kappa_N( r\otimes 1)\cdot\tilde\kappa_N\big(1\otimes j_\alpha(\nu)\big)\\
&=\epsilon_{\tilde\kappa,M}(b)
\cdot\epsilon_{\tilde\kappa,p}(-1)\cdot\kappa_N(r)\cdot\int_{\Z_p^\times}\tilde\kappa_N(t)dj_\alpha(\nu)\end{split}\]
and thus, noticing that $\tilde\kappa_N$ restricted to $\Z_p^\times$ is $\tilde\kappa_N(t)=
\epsilon_{\tilde\kappa,p}(t)t^{n_{\tilde\kappa}}$, 
we have 
\[\tilde\kappa_N\big(J_\alpha(r\otimes\nu)\big)=\epsilon_{\tilde\kappa,M}(b)
\cdot\epsilon_{\tilde\kappa,p}(-1)\cdot \kappa_N(r)\int_{\Z_p\times\Z_p^\times}
\epsilon_{\tilde\kappa,p}(\tilde Q_\alpha(x,y))\tilde Q_\alpha(x,y)^{n_{\tilde\kappa}/2}d\nu.\]
Recalling \eqref{Q}, and viewing $a,b,c$ as elements in $\Z_p$, 
we have, for $(x,y)\in\Z_p\times\Z_p^\times$,  
\[\epsilon_{\tilde\kappa,p}\big(\tilde Q_\alpha(x,y)\big)=\epsilon_{\tilde\kappa,p}(-by^2)=\epsilon_{\tilde\kappa,p}(-b)
\epsilon_{\tilde\kappa,p}(y^2)=\epsilon_{\tilde\kappa,p}(-b)\epsilon_{\tilde\kappa,p}^2(y)=
\epsilon_{\tilde\kappa,p}(-b)\epsilon_{\kappa,p}(y).\] 
Thus, since $\epsilon_{\tilde\kappa}(-1)^2=1$, we get: 
\[\tilde\kappa_N\big(J_\alpha(r\otimes\nu)\big)=\kappa_N(r)\cdot 
\epsilon_{\tilde\kappa,M}(b)
\cdot\epsilon_{\tilde\kappa,p}(b)\cdot\rho_\kappa(\nu)(\tilde Q_\alpha^{n_{\tilde\kappa}/2})=\eta_{\epsilon_\kappa}(\alpha)\cdot\langle \nu,\alpha\rangle_{\kappa_N}\]
where for the last equality use \eqref{eta} and \eqref{pairing}.\end{proof}

Define \[\W_{\mathcal R_N}:=\W\otimes_{\mathcal O[\![\Z_p^\times]\!]}\mathcal R_N,\] 
the structure of $\mathcal O[\![\Z_p^\times]\!]$-module of $\mathcal R_N$ being that induced by 
the composition of the two maps $\mathcal O[\![\Z_p^\times]\!]\rightarrow\Lambda_N\rightarrow\mathcal R_N$ 
described above. There is a canonical map 
\[\vartheta:\W_{\mathcal R_N}\longrightarrow H^1(\Gamma_0,\D_{\mathcal R_N})\] described as follows: 
if $\nu_\gamma$ is a representative of an element $\nu$ in $\W$ and $r\in \mathcal R_N$, then 
$\vartheta(\nu\otimes r)$ is represented by the cocycle $\nu_\gamma\otimes r$. 

For $\nu\in \W_{\mathcal R_N}$ represented by $\nu_\gamma$ and $\xi\geq 1$ an integer, define 
\[\theta_{\xi}(\nu):= \sum_{\calC\in R(\Gamma_1), q(\calC)=\xi} \frac{J_{\alpha_\calC}(\nu_{\gamma_{\alpha_\calC}})}{\gert_{\alpha_\calC}}.\]

\begin{dfn} For $\nu\in\W_{\mathcal R_N}$, 
the formal Fourier expansion 
\[
\Theta(\nu):=\sum_{\xi\geq 1}\theta_{\xi}(\nu)q^\xi
\] in $\mathcal R_N[\![q]\!]$ is called the {$\Lambda$-adic Shimura-Shintani-Waldspurger lift} of $\nu$. 
For any $\tilde\kappa\in\widetilde{\mathcal X}^{\rm arith} $, 
the formal power series expansion  
\[\Theta(\nu)(\tilde\kappa_N):=\sum_{\xi\geq 1}
\tilde\kappa_N\big(\theta_{\xi}(\nu)\big)q^\xi\] is called the {$\tilde\kappa$-specialization} of $\Theta(\nu)$. 
\end{dfn}

There is a natural map
\[\W_{\mathcal R}\longrightarrow\W_{\mathcal R_N}\] taking $\nu\otimes r$ to itself (use that $\mathcal R$ has 
a canonical map to $\mathcal R_N\simeq\mathcal R[\Delta]$, as described above). So, for any choice of sign, 
$\Phi^\pm\in \W_{\mathcal R}$ will be viewed as an element in $\W_{\mathcal R_N}$. 

From now on we will use the following notation. Fix $\tilde\kappa_0\in\widetilde{\mathcal X}^{\rm arith} $ and put 
$\kappa_0:=p(\tilde\kappa_0)\in\mathcal X^{\rm arith} $. 
Recall the neighborhood $\mathcal U_{0} $ of $\kappa_0$ in Theorem \ref{thm-LV}. 
Define $\widetilde{\mathcal U}_0 :=p^{-1}(\mathcal U_0 )$ and  
\[\widetilde{\mathcal U}^{\rm arith}_0 :=\widetilde{\mathcal U}_0 \cap 
\widetilde{\mathcal X}^{\rm arith} .\] 
For each $\tilde\kappa\in\widetilde{\mathcal U}_0^{\rm arith}$ put 
$\kappa=p(\tilde\kappa)\in \mathcal U_0^{\rm arith} $.
Recall that if $(\epsilon_{\tilde\kappa},k_{\tilde\kappa})$ 
is the signature of $\tilde\kappa$, then 
$(\epsilon_\kappa,k_\kappa):=(\epsilon_{\tilde\kappa}^2,2k_{\tilde\kappa})$ is that of $\kappa_0$. 
For any $\kappa:=p(\tilde\kappa)$ as above, we may consider the modular form 
\[f_\kappa^{\rm JL}\in S_{k_\kappa}(\Gamma_{r_\kappa},\epsilon_{\kappa})\] 
and its Shimura-Shintani-Waldspurger lift \[h_\kappa=\sum_\xi a_\xi(h_\kappa)q^\xi\in S_{k_\kappa+1/2}(4Np^{r_\kappa},\chi_\kappa), \quad
\text{where }\chi_\kappa(x):=\epsilon_{\tilde\kappa}(x)\left(\frac{-1}{x}\right)^{k_\kappa},\]
normalized as in \eqref{SS1} and \eqref{SS2}. 
For our fixed $\kappa_0$, recall the elements $\Phi:=\Phi^{+}$ chosen 
as in Theorem \ref{thm-LV} and define $\phi_\kappa:=\phi_\kappa^+$ and 
$\Omega_\kappa:=\Omega_\kappa^+$ for $\kappa\in\mathcal U_0^{\rm arith}$. 

\begin{prop}\label{prop-final-1}For all $\tilde\kappa\in\widetilde{\mathcal U}_{0}^{\rm arith}$ such that $r_{\kappa}=1$ 
we have 
\[\tilde\kappa_N\big(\theta_{\xi}(\Phi )\big)={\Omega_\kappa}\cdot a_\xi(h_\kappa)
\quad\text{and}\quad \Theta(\Phi )(\tilde\kappa_N)= {\Omega_\kappa}\cdot h_\kappa.\]
\end{prop}

\begin{proof}
By Lemma \ref{lemma3.2} we have 
\[
\tilde\kappa_N\big(\theta_{\xi}(\Phi)\big)= 
\sum_{\calC\in R(\Gamma_{1}), q(\calC)=\xi} \frac{\eta_{\epsilon_{\tilde\kappa}}(\alpha_\mathcal C)}{\gert_{\alpha_\calC}}
\rho_{\kappa_N}(\Phi )(\tilde Q_{\alpha_\mathcal C}^{n_{\tilde\kappa}/2}). 
\] 
Using Theorem \ref{thm-LV}, we get 
\[
\tilde\kappa_N\big(\theta_{\xi}(\Phi )\big)= 
\sum_{\calC\in R(\Gamma_{1}), q(\calC)=\xi} \frac{\eta_{\epsilon_{\tilde\kappa}}(\alpha_\mathcal C)
\cdot \Omega _\kappa}{\gert_{\alpha_\calC}}
\phi _\kappa(\tilde Q_{\alpha_\mathcal C}^{k_\kappa-1}).
\] Now \eqref{SS1} shows the statement on 
$\tilde\kappa_N(\theta_{\xi}(\Phi ))$,  
while that on $\Theta(\Phi )(\tilde\kappa_N)$ is a formal consequence of the
previous one.
\end{proof}

\begin{cor}\label{coro-final} Let $a_p$ denote the image of the Hecke operator $T_p$ in $\mathcal R$. Then 
$\Theta(\Phi)|T_p^2=a_p\cdot \Theta(\Phi)$. 
\end{cor}

\begin{proof} For any $\kappa\in\mathcal X^{\rm arith}$, let $a_p(\kappa):=\kappa(T_p)$, which is a $p$-adic unit 
by the ordinarity assumption. For all $\tilde\kappa\in \widetilde{\mathcal U}_0^{\rm arith}$ with $r_\kappa=1$, we have \[\Theta(\Phi)(\tilde\kappa_N)|T_p^2={\Omega_\kappa}\cdot h_\kappa|T_p^2 =a_p(\kappa)\cdot{\Omega_\kappa}\cdot h_\kappa=
a_p(\kappa)\cdot\Theta(\Phi)(\tilde\kappa_N).\] Consequently, 
\[\tilde\kappa_N\big(\theta_{\xi p^2}(\Phi)\big)=a_p(\kappa)\cdot\tilde\kappa_N\big(\theta_{\xi}(\Phi)\big)\]
for all $\tilde\kappa$ such that $r_\kappa=1$. Since this subset is dense in $\widetilde{\mathcal X}_N$, we conclude 
that $\theta_{\xi p^2}(\Phi) =a_p \cdot\theta_{\xi}(\Phi)$ and so $\Theta(\Phi)|T_p^2=a_p\cdot \Theta(\Phi)$. \end{proof}

For any integer $n\geq 1$ and any quadratic form $Q$ with coefficients in $F$,  
write $[Q]_n$ for the class of $Q$ modulo the action of $i_F(\Gamma_n)$.    
Define $\mathcal F_{n,\xi}$ to be the subset of the $F$-vector space of quadratic forms with coefficients in $F$
consisting of quadratic forms 
$\tilde Q_{\alpha}$ such that $\alpha\in V^*\cap\mathcal O_{B,n}$ and $-{\rm nr}(\alpha)=\xi$. 
Writing $\delta_{\tilde Q_\alpha}$ for the discriminant of $Q_\alpha$,
the above set can be equivalently described as 
\[\mathcal F_{n,\xi}:=
\{\tilde Q_{\alpha} |\,\alpha\in V^*\cap\mathcal O_{B,n},\, \delta_{\tilde Q_\alpha}=Np^n\xi\}.\]
Define $\mathcal F_{n,\xi}/\Gamma_n$ to be the 
set $\{[\tilde Q_\alpha]_n|\,\tilde Q_\alpha\in\mathcal F_{n,\xi}\}$ of equivalence classes of $\mathcal F_{n,\xi}$ under 
the action of $i_F(\Gamma_n)$. 
A simple computation shows that $Q_{g^{-1}\alpha g}=Q_{\alpha}|g$ for 
all $\alpha\in V^*$ and all $g\in\Gamma_n$,
and thus we find 
\[\mathcal F_{n,\xi}/\Gamma_n=\{[\tilde Q_{\mathcal C_\alpha}]_n|\, \mathcal C\in R(\Gamma_n), \, 
\delta_{\tilde Q_\alpha}=Np^n\xi \}.\] We also note that, in the notation of \S \ref{sec-SSL}, if $f$ has weight character $\psi$, defined modulo
$Np^n$, and level $\Gamma_n$, the Fourier coefficients $a_\xi(h)$ of the Shimura-Shintani-Waldspurger lift $h$ of $f$ are given 
by 
\begin{equation}\label{Fourier-last}
a_\xi(h)=\sum_{[Q]\in\mathcal F_{n,\xi}/\Gamma_n}\frac{\psi(Q)}{\gert_Q}\phi^+_f\big(Q(z)^{k-1}\big)\end{equation} and, 
if $Q=\tilde Q_\alpha$, we put ${\psi(Q)}:=\eta_\psi(b_\alpha)$ and ${\gert_Q}:=\gert_{\alpha}$. Also, if we let 
\[\mathcal F_n/\Gamma_n:=\coprod_{\xi}\mathcal F_{n,\xi}/\Gamma_n\] we can write 
\begin{equation}\label{h-last}
h=\sum_{[Q]\in\mathcal F_n/\Gamma_n } \frac{\psi(Q)}{\gert_Q}\phi^+_f\big(Q(z)^{k-1}\big)q^{\delta_Q/(Np^n)}.\end{equation}

Fix now an integer $m\geq 1$ and let $n\in\{1,m\}$. 
For any $t\in (\Z/p^n\Z)^\times$ and any integer $\xi\geq 1$, define 
$\mathcal F_{n,\xi,t}$ to be the subset of $\mathcal F_{n,\xi}$ consisting of forms such that  
$Np^nb_\alpha\equiv t\mod Np^m $. Also, define $\mathcal F_{n,\xi,t}/\Gamma_n$ to be the 
set of equivalence classes of $\mathcal F_{n,\xi,t}$ under 
the action of $i_F(\Gamma_n)$. 
If $\alpha\in V^*\cap\mathcal O_{B,m}$ and 
$i_F(\alpha)=\smallmat abc{-a}$, then 
\begin{equation}\label{tilde-Q}\tilde Q_{\alpha}(x,y)=Np^ncx^2-2Np^naxy-Np^nby^2\end{equation}  from which we see 
that there is an inclusion 
$\mathcal F_{m,\xi,t}\subseteq \mathcal F_{1,\xi p^{m-1},t}$.
If $\tilde Q_\alpha$ and $\tilde Q_{\alpha'}$ 
belong to $\mathcal F_{m,\xi,t}$, and $\alpha'=g\alpha g^{-1}$ for some $g\in\Gamma_m$, then, since 
$\Gamma_m\subseteq\Gamma_1$, we see that $\tilde Q_\alpha$ and $\tilde Q_{\alpha'}$ 
represent the same class in $\mathcal F_{1,\xi p^{m-1},t}/\Gamma_1$.
This shows that  $[\tilde Q_{\alpha}]_m\mapsto[\tilde Q_{\alpha}]_1$ 
gives a well-defined map  \[\pi_{m,\xi,t}:\mathcal F_{m,\xi,t}/\Gamma_m\longrightarrow \mathcal F_{1, \xi p^{m-1},t}/\Gamma_1. 
\]  

\begin{lem}\label{lemma-pi} The map $\pi_{m,\xi,t}$ is bijective. 
\end{lem}

\begin{proof} We first show the injectivity. 
For this, suppose $\tilde Q_{\alpha}$ and $\tilde Q_{\alpha'}$ are in 
$\mathcal F_{m,\xi,t}$
and $[\tilde Q_\alpha]_{1}=[\tilde Q_{\alpha'}]_{1}$. 
So there exists $g=\smallmat\alpha\beta\gamma\delta$ in $i_F(\Gamma_1)$ such that 
such that $\tilde Q_{\alpha}=\tilde Q_{\alpha'}|g$. If $\tilde Q_{\alpha}=cx^2-2axy-by^2$, and easy computation shows that 
$\tilde Q_{\alpha'}=c'x^2-2a'xy-b'y^2$ with 
\[c'=c\alpha^2-2a\alpha\gamma-b\gamma^2\]
\[a'=-c\alpha\beta+a\beta\gamma+a\alpha\delta+b\gamma\delta\]
\[b'=-c\beta^2+2a\beta\delta+b\delta^2.\] The first condition shows that 
$\gamma\equiv 0\mod Np^m$. 
We have $b\equiv b'\equiv t\mod Np^m$, so $\delta^2\equiv1\mod Np^m$. 
Since $\delta\equiv 1\mod Np$, we see that $\delta\equiv1\mod Np^m$ too. 

We now show the surjectivity. For this, fix $[\tilde Q_{\alpha_\mathcal C}]_1$ in the target of $\pi$, 
and choose a representative 
\[\tilde Q_{\alpha_\mathcal C}= cx^2-2axy-by^2
\] (recall $Mp^m\xi| \delta_{\tilde Q_{\alpha_\mathcal C}}$, $Np| c$, $Np| a$, and $b\in \mathcal O_F^\times$, the last condition due to $\eta_\psi(\alpha_\mathcal C)\neq 0$). 
By the Strong Approximation Theorem, we can find $\tilde g\in \Gamma_1$ such that 
\[i_\ell(\tilde g)\equiv \mat 10{ab^{-1}}1\mod Np^m\] for all $\ell| Np$.  
Take $g:=i_F(\tilde g)$, and put $\alpha:=g^{-1}\alpha_\mathcal C g$. An easy computation, using the expressions for 
$a',b',c'$ in terms of $a,b,c$ and $g=\smallmat \alpha\beta\gamma\delta$ as above, shows 
that $\alpha\in V^*\cap  \mathcal O_{B,m}$, $\eta_\psi(\alpha)=t$ and $\delta_{\tilde Q_\alpha}=Np^m\xi$, and 
it follows that $\tilde Q_{\alpha}\in\mathcal F_{m,\xi,t}$. Now \[\pi\big([\tilde Q_{\alpha}]_m\big)=[\tilde Q_{\alpha}]_1=
[\tilde Q_{g^{-1}\alpha_\mathcal C g}]_1=[\tilde Q_{\alpha_\mathcal C}]_1\] 
where the last equality follows because $g\in \Gamma_1$.
 \end{proof}

\begin{prop}\label{prop-final-2} For all $\tilde\kappa\in\widetilde{\mathcal U}_{0}^{\rm arith}$
we have 
\[\Theta(\Phi )(\tilde\kappa_N)|T_p^{r_\kappa-1}= {\Omega_\kappa}\cdot h_\kappa.\]\end{prop}

\begin{proof}
For $r_\kappa=1$, this is Proposition \ref{prop-final-1} above, so we may assume $r_\kappa\geq 2$. 
As in the proof of Proposition \ref{prop-final-1}, combining Lemma \ref{lemma3.2} and Theorem \ref{thm-LV} we get 
\[\Theta(\Phi )(\tilde\kappa_N)= 
\sum_{\xi\geq 1}\left(\sum_{\calC\in R(\Gamma_{1}), q(\calC)=\xi} \frac{\eta_{\epsilon_{\tilde\kappa}}(\alpha_\mathcal C)
\cdot \Omega _\kappa}{\gert_{\alpha_{\mathcal C}}}
\phi _\kappa(\tilde Q_{\alpha_\mathcal C}^{k_\kappa-1})\right)q^\xi
\] which, by \eqref{Fourier-last} and \eqref{h-last} above we may rewrite as 
\[\Theta(\Phi )(\tilde\kappa_N)= \sum_{[Q]\in\mathcal F_1/\Gamma_1}\frac{{\epsilon_{\tilde\kappa}}(Q)
\cdot \Omega _\kappa}{\gert_{Q}}
\phi _\kappa(Q^{k_\kappa-1})q^{\delta_Q/(Np)}
\] By definition of the action of $T_p$ on power series, we have 
\[\Theta(\Phi )(\tilde\kappa_N)|T_p^{r_\kappa-1}=  \sum_{[Q]\in\mathcal F_1/\Gamma_1,\,p^{r_\kappa}|\delta_Q}\frac{{\epsilon_{\tilde\kappa}}(Q)
\cdot \Omega _\kappa}{\gert_{Q}}
\phi _\kappa(Q^{k_\kappa-1})q^{\delta_Q/(Np^{r_\kappa})}.\] Setting 
$\mathcal F_{n,t}/\Gamma_n:=\coprod_{\xi\geq 1}\mathcal F_{n,t,\xi}/\Gamma_n$ 
for $n\in\{1,r_\kappa\}$, Lemma \ref{lemma-pi} shows that  
\[\mathcal F^*_{1,t}:=\{[Q]\in\mathcal F_{1,t}/\Gamma_{1,t}\text{ such that } p^{r_\kappa}| \delta_Q\}\text{ is equal to }\mathcal F_{r_\kappa,t}.\]
Therefore, splitting the above sum over $t\in(\Z/Np^{r_\kappa}\Z)^\times$, we get 
\[\begin{split}\Theta(\Phi )(\tilde\kappa_N)|T_p^{r_\kappa-1}&=
\sum_{t\in(\Z/p^{r_\kappa-1}\Z)^\times} \sum_{[Q]\in\mathcal F_{1,t}^*}\frac{{\epsilon_{\tilde\kappa}}(Q)
\cdot \Omega _\kappa}{\gert_{Q}}
\phi _\kappa(Q^{k_\kappa-1})q^{\delta_Q/(Np^{r_\kappa})} \\
&=\sum_{t\in(\Z/p^{r_\kappa-1}\Z)^\times} \sum_{[Q]\in\mathcal F_{m,t}/\Gamma_{m}}\frac{{\epsilon_{\tilde\kappa}}(Q)
\cdot \Omega _\kappa}{\gert_{Q}}
\phi _\kappa(Q^{k_\kappa-1})q^{\delta_Q/(Np^{r_\kappa})}\\
&=\sum_{[Q]\in\mathcal F_{m}/\Gamma_{m}}\frac{{\epsilon_{\tilde\kappa}}(Q)
\cdot \Omega _\kappa}{\gert_{Q}}
\phi _\kappa(Q^{k_\kappa-1})q^{\delta_Q/(Np^{r_\kappa})}.
\end{split}\]
Comparing this expression with \eqref{h-last} gives the result. 
\end{proof}

We are now ready to state the analogue of \cite[Thm. 3.3]{St}, which is our main result. For the reader's convenience, 
we briefly recall the notation appearing 
below. We denote by $\mathcal X$ the points of the ordinary Hida Hecke algebra, and by $\mathcal X^{\rm arith}$ its arithmetic points. 
For $\kappa_0\in\mathcal X^{\rm arith}$, we denote
by $\mathcal U_0$ 
the $p$-adic neighborhood of $\kappa_0$ appearing in the statement of Theorem \ref{thm-LV} and put 
$\mathcal U_0^{\rm arith}:=\mathcal U_0\cap\mathcal X^{\rm arith}$. 
We also denote by  
$\Phi=\Phi^+\in\W_{\mathcal R}^\ord$ the cohomology class appearing in Theorem \ref{thm-LV}. 
The points $\widetilde{\mathcal X}$ of the metaplectic Hida Hecke algebra 
defined in \S\ref{metaplectic} are equipped with a canonical map $p:\widetilde{\mathcal X}^{\rm arith}\rightarrow{\mathcal X}^{\rm arith}$ on arithmetic points. Let $\widetilde{\mathcal U}^{\rm arith}_0 :=\widetilde{\mathcal U}_0 \cap 
\widetilde{\mathcal X}^{\rm arith}$. For each $\tilde\kappa\in\widetilde{\mathcal U}_0^{\rm arith}$ put 
$\kappa=p(\tilde\kappa)\in \mathcal U_0^{\rm arith} $.
Recall that if $(\epsilon_{\tilde\kappa},k_{\tilde\kappa})$ 
is the signature of $\tilde\kappa$, then 
$(\epsilon_\kappa,k_\kappa):=(\epsilon_{\tilde\kappa}^2,2k_{\tilde\kappa})$ is that of $\kappa_0$. 
For any $\kappa:=p(\tilde\kappa)$ as above, we may consider the modular form 
\[f_\kappa^{\rm JL}\in S_{k_\kappa}(\Gamma_{r_\kappa},\epsilon_{\kappa})\] 
and its Shimura-Shintani-Waldspurger lift \[h_\kappa=\sum_\xi a_\xi(h_\kappa)q^\xi\in S_{k_\kappa+1/2}(4Np^{r_\kappa},\chi_\kappa), \quad
\text{where }\chi_\kappa(x):=\epsilon_{\tilde\kappa}(x)\left(\frac{-1}{x}\right)^{k_\kappa},\]
normalized as in \eqref{SS1} and \eqref{SS2}. Finally, for $\tilde\kappa\in\widetilde{\mathcal X}^{\rm arith}$, we denote by $\tilde\kappa_N$ 
its extension to the metaplectic Hecke algebra $\widetilde{\mathcal R}_N$ defined in \S \ref{metaplectic}. 

\begin{thm}\label{main} Let $\kappa_0 \in \calX^{\rm arith} $.
Then there exists a choice of $p$-adic periods $\Omega_{\kappa}$ for $\kappa\in \mathcal U_{0}$ such that the {$\Lambda$-adic Shimura-Shintani-Waldspurger lift} of $\Phi$
\[
\Theta(\Phi):=\sum_{\xi\geq 1}\theta_{\xi}(\Phi)q^\xi
\] in $\mathcal R_N[\![q]\!]$ has the following properties:

\begin{enumerate}
\item $\Omega_{\kappa_0} \neq 0$.

\item

For any $\tilde\kappa\in\widetilde{\mathcal U}_0^{\rm arith} $, 
 the {$\tilde\kappa$-specialization} of $\Theta(\Phi)$ 
\[\Theta(\nu)(\tilde\kappa_N):=\sum_{\xi\geq 1}
\tilde\kappa \big(\theta_{\xi}(\Phi)\big)q^\xi \text{ belongs to } 
S_{k_\kappa+1/2}(4Np^{r_\kappa},\chi_\kappa'),\] 
where
$\chi_\kappa'(x):=\chi_\kappa(x)\cdot\left(\frac{p}{x}\right)^{k_\kappa-1}$, and satisfies 
\[ \Theta(\Phi)(\tilde\kappa_N) = \Omega_{\kappa} \cdot h_\kappa | T_p^{1-r_\kappa} .\]  
\end{enumerate} \end{thm} 

\begin{proof} The elements $\Omega_\kappa$ are those $\Omega_\kappa^+$ appearing in Theorem \ref{thm-LV}, which we used in 
Propositions \ref{prop-final-1} and \ref{prop-final-2} above, so (1) is clear. 
Applying $T_p^{r_\kappa-1}$ to the formula of Proposition \ref{prop-final-2}, using Corollary \ref{coro-final} and applying 
$a_p(\kappa)^{1-r_\kappa}$ on both sides gives 
\[\Theta(\Phi )(\tilde\kappa_N)= a_p(\kappa)^{1-r_\kappa}{\Omega_\kappa}\cdot h_\kappa|T_p^{r_\kappa-1}.\] 
By \cite[Prop. 1.9]{Shimura}, each application of $T_p$ has the effect of multiplying the character by $\left(\frac{p}{\cdot}\right)$, hence 
\[h_\kappa':=h_\kappa|T_p^{r_\kappa-1}\in
 S_{k_\kappa+1/2}(4Np^{r_\kappa},\chi_\kappa')\]
 with $\chi'_\kappa$ as in the statement. This gives the first part of (2), while   
the last formula follows immediately from Proposition \ref{prop-final-2}.\end{proof}

\begin{rmk}
Theorem \ref{thm-intro} is a direct consequence of Theorem \ref{main}, as we briefly show below. 

Recall the embedding $\Z^{\geq 2}\hookrightarrow\Hom(\Z_p^\times,\Z_p^\times)$ which 
sends $k\in\Z^{\geq 2}$ to the character $x\mapsto x^{k-2}$. 
Extending characters by $\mathcal O$-linearity gives 
a map 
\[
\Z^{\geq 2}\longmono \mathcal X(\Lambda):=\Hom_{\mathcal O\text{-alg}}^{\rm cont}(\Lambda,\bar\Q_p).
\] 
We denote by $k^{(\Lambda)}$ the image of $k\in \Z^{\geq 2}$ in $\mathcal X(\Lambda)$ 
via this embedding. 
We also denote by $\varpi:\mathcal X\rightarrow \mathcal X(\Lambda)$ the finite-to-one map 
obtained by restriction of homomorphisms to $\Lambda$. Let 
$k^{(\mathcal R)}$ be a point 
in $\mathcal X$ of signature $(k,1)$ such that 
$\varpi(k^{(\mathcal R)})=k^{(\Lambda)}$.
A well-known result by Hida (see \cite[Cor. 1.4]{Hida-Galois}) 
shows that $\mathcal R/\Lambda$ is unramified at $k^{(\Lambda)}$. 
As shown in \cite[\S 1]{St}, this implies that 
there is a section $s_{k^{(\Lambda)}}$ of 
$\varpi$ 
which is defined on a neighborhood 
$\mathcal U_{k^{(\Lambda)}}$ 
of $k^{(\Lambda)}$ in $\mathcal X(\Lambda)$ and sends $k^{(\Lambda)}$ to 
$k^{(\mathcal R)}$.  

Fix now $k_0$ as in the statement of Theorem \ref{thm-intro},
corresponding to a cuspform $f_0$ of weight $k_0$ with trivial 
character. The form $f_0$ corresponds to an arithmetic character 
$k_0^{(\mathcal R)}$ of signature $(1,k_0)$ belonging to $\mathcal X$. 
Let  
$\mathcal U_0'$ be the inverse image of $\mathcal U_0$ under the 
section $s_{k^{(\Lambda)}_0}^{-1}$, 
where $\mathcal U_0\subseteq\mathcal X$ is the neighborhood
of $k_0^{(\mathcal R)}$ in Theorem \ref{main}. 
Extending scalars by $\mathcal O$ gives, as above, an injective continuos map 
$
\Hom(\Z_p^\times,\Z_p^\times)\hookrightarrow\mathcal X(\Lambda), 
$
and we let 
$U_0$ be any neighborhood of the character $x\mapsto x^{k_0-2}$ which maps to 
$\mathcal U_0'$ and is contained in the residue class of $k_0$ modulo $p-1$. 
Composing this map with the section 
$\mathcal U_0'\hookrightarrow \mathcal U_0$
gives a continuous injective map  
\[\varsigma:U_0\,\longmono\, \mathcal U_0'\,\longmono\, \mathcal U_0\]
which takes $k_0$ to $k_0^{(\mathcal R)}$, since by 
construction the image of $k_0$ by the first map is $k^{(\Lambda)}_0$. We also note that, more generally, 
$\varsigma(k)=k^{(\mathcal R)}$ because by construction $\varsigma(k)$ restricts to $k^{(\Lambda)}$ 
and its signature is 
$(1,k)$, since the character of $\varsigma(k)$ is trivial. To show the last assertion, 
recall that the character of $\varsigma(k)$
is $\psi_k\cdot\psi_\mathcal R\cdot\omega^{-k}$, and  
note that $\psi_k$ is trivial because 
$k^{(\Lambda)}(x)=x^{k-1}$, and 
$\psi_\mathcal R\cdot\omega^{-k}$ is trivial because the same is true for $k_0$ 
and $k\equiv k_0$ modulo $p-1$. 
In other words, arithmetic points in $\varsigma(U_0)$ correspond to 
cuspforms with trivial character. This is the Hida family of forms with trivial character that we 
considered in the Introduction.

We can now prove Theorem \ref{thm-intro}.  
For all $k\in U_0\cap \Z^{\geq 2}$, put $\Omega_k:=\Omega_{k^{(\mathcal R)}}$
and define $\Theta:=\Theta(\Phi)\circ\varsigma$ 
with $\Phi$ as in Theorem \ref{main} for $\kappa_0=k^{(\mathcal R)}_0$. 
Applying Theorem \ref{main} to $k_0^{(\mathcal R)}$, and restricting to $\varsigma(U_0)$, 
shows that $U_0$, $\Omega_k$ and $\Theta$ 
satisfy the conclusion of Theorem \ref{thm-intro}.\end{rmk}
 
\begin{rmk}
For $\tilde\kappa\in\widetilde{\mathcal U}_{0}^{\rm arith} $ of signature $(\epsilon_{\tilde\kappa}, k_{\tilde\kappa})$ 
with $r_{\tilde\kappa}=1$ 
as in the above theorem, 
$h_\kappa$ is trivial if $(-1)^{k_{\tilde\kappa}}\neq\epsilon_{\tilde\kappa}(-1)$. 
However, since $\phi_{\kappa_0}\neq 0$, it follows that 
$h_{\kappa_0}$ is not trivial as long as 
the necessary condition  $(-1)^{k_{0}}=\epsilon_{0}(-1)$ is verified. 
\end{rmk}

\begin{rmk}
This result can be used to produce a quaternionic
$\Lambda$-adic version of the Saito-Kurokawa lifting, following closely the arguments in \cite[Cor. 1]{LN}.
\end{rmk}


\begin{thebibliography}{99}

\bibitem{CS} J. Coates, R. Sujatha, 
{\it Cyclotomic fields and zeta values.} Springer Monographs in Mathematics. Springer-Verlag, Berlin, 2006. 

\bibitem{DT} H. Darmon, G. Tornaria, Stark-Heegner points and the Shimura correspondence.
\emph{Compositio Math.}, {\bf 144} (2008) 1155-1175. 

\bibitem{GS} R. Greenberg, G. Stevens, 
$p$-adic $L$-functions and $p$-adic periods of modular forms. 
\emph{Invent. Math.} {\bf 111} (1993), no. 2, 407--447.

\bibitem{Ko} Koblitz, N., {\it Introduction to elliptic curves and modular forms.}
Graduate Texts in Mathematics, {\bf 97}. Springer-Verlag, New York, 1984. viii+248 pp.

\bibitem{Ko2} W. Kohnen, 
Fourier coefficients of modular forms of half-integral weight. \emph{Math. Ann.} {\bf 271} (1985), no. 2, 237--268. 

\bibitem{Hida-Galois} Hida H., 
Galois representations into $\GL_2(\Z_p[[X]])$ attached to ordinary cusp forms. 
\emph{Invent. Math.} {\bf 85} (1986), no. 3, 545--613. 

\bibitem{Hida3} Hida H., On $\Lambda$-adic forms of half-integral weight for $\SL(2)/\Q$. Number Theory (Paris 1992-3). Lond. Math. Soc. Lect. Note Ser.

\bibitem{LN} M. Longo, M.-H. Nicole, The Saito-Kurokawa lifting and Darmon points, to appear in Math. Ann. 
DOI 10.1007/s00208-012-0846-5

\bibitem{LV1} M. Longo, S. Vigni, A note on control theorems for quaternionic Hida families of modular forms, Int. J. Number Theory, 2012 (to appear).

\bibitem{LV2} M. Longo, S. Vigni, The rationality of quaternionic Darmon points over genus fields of real quadratic fields, preprint 2011. 

\bibitem{NP} J. Nekov\'a\v{r}, A. Plater, 
On the parity of ranks of Selmer groups. 
\emph{Asian J. Math.} {\bf 4} (2000), no. 2, 437--497. 

\bibitem{Pa} J. Park, $p$-adic family of half-integral weight modular forms via overconvergent Shintani lifting
{\em Manuscripta Mathematica}, Volume {\bf 131}, 3-4, 2010, 355-384. 

\bibitem{Po} A. Popa, Central values of Rankin L-series over real quadratic fields. \emph{Compos. Math.} {\bf 142} (2006), no. 4, 811--866.

\bibitem{Pr3} K. Prasanna, Integrality of a ratio of Petersson norms and level-lowering congruences. 
\emph{Ann. of Math.} (2) {\bf 163} (2006), no. 3, 901--967.

\bibitem{Pr} K. Prasanna, Arithmetic properties of the Shimura-Shintani-Waldspurger correspondence. With an appendix by Brian Conrad. \emph{Invent. Math.} {\bf 176} (2009), no. 3, 521--600. 

\bibitem{Pr2} K. Prasanna, 
On the Fourier coefficients of modular forms of half-integral weight. 
\emph{Forum Math.} {\bf 22} (2010), no. 1, 153--177.

\bibitem{Ra} Ramsey, N., The overconvergent Shimura lifting,
  {\em  Int. Math. Res. Not.}, 2009, no. {\bf 2}, p. 193-220.

\bibitem{Shimura} 
G. Shimura, On modular forms of half integral weight. 
\emph{Ann. of Math.} (2) {\bf 97} (1973), 440--481. 

\bibitem{Sh} G. Shimura, 
The periods of certain automorphic forms of arithmetic type. 
\emph{J. Fac. Sci. Univ. Tokyo Sect. IA Math.} {\bf 28} (1981), no. 3, 605-632 (1982).

\bibitem{Shin} T. Shintani, 
On construction of holomorphic cusp forms of half integral weight. 
\emph{Nagoya Math. J.} {\bf 58} (1975), 83--126.

\bibitem{St} G. Stevens,
$\Lambda$-adic modular forms of half-integral weight and a 
$\Lambda$-adic Shintani lifting. Arithmetic geometry (Tempe, AZ, 1993), 
129--151,
Contemp. Math., {\bf 174}, Amer. Math. Soc., Providence, RI, 1994. 

\bibitem{Wa} J.-L. Waldspurger, 
Correspondances de Shimura et quaternions.
\emph{Forum Math.} {\bf 3} (1991), no. 3, 219--307.

\end{thebibliography}
\end{document}